\documentclass[12pt]{article}

\usepackage[paper=letterpaper,margin=0.7in,twoside=false,includehead]{geometry}

\usepackage{amsfonts,amsmath,amsthm, amssymb} 
 
\usepackage{hyperref}
\usepackage[capitalize]{cleveref}
\usepackage[T1]{fontenc}
\usepackage{mathtools}
\usepackage{verbatim}
\usepackage{caption}

\DeclarePairedDelimiter\floor{\lfloor}{\rfloor}

\DeclarePairedDelimiter\abs{|}{|}
\DeclarePairedDelimiter\parens{(}{)}
\DeclarePairedDelimiter\set{\{}{\}}

\usepackage{tikz,tkz-berge}

\DeclareDocumentCommand{\K}{O{t}}{\mathcal{K}^{#1}}

\DeclareMathOperator{\CT}{CT}

\newtheorem{thm}{Theorem}[section]
\newtheorem{lem}[thm]{Lemma}
\newtheorem{cor}[thm]{Corollary}
\newtheorem{prop}[thm]{Proposition}

\theoremstyle{definition}
\newtheorem{defn}[thm]{Definition}
\newtheorem{problem}{Problem}

\newcommand{\D}{\Delta}
\newcommand{\cH}{\mathcal{H}}
\newcommand{\T}{\mathrm{T}}
\newcommand{\N}{\mathcal{N}}
\newcommand{\w}{\omega}
\newcommand{\dom}{\mathrm{dom}}
\newcommand{\Dom}{\mathrm{Dom}}
\newcommand{\ex}{\mathrm{ex}}
\newcommand{\mex}{\mathrm{mex}}

\begin{document}

\pagestyle{plain}

\title{Maximizing subgraph density in graphs of bounded degree and clique number}
\author{R. Kirsch}
\date{August 14, 2025}

\maketitle
 
\abstract{
We asymptotically determine the maximum density of subgraphs isomorphic to $H$, where $H$ is any graph containing a dominating vertex, in graphs $G$ on $n$ vertices with bounded maximum degree and bounded clique number. That is, we asymptotically determine the constant $c=c(H,\D,\w)$ such that $\ex(n,H,\set{K_{1,\D+1},K_{\w+1}})=(1-o_n(1))cn$ where $\w$ is sufficiently large. 

Following recent interest in the corresponding parameter $\mex(m,H,F)$ where where we fix the number of edges $m$ instead of the number of vertices $n$ of the graph, we determine the asymptotics of $\mex(m,H,\set{K_{1,1,\D+1},K_{\w+1}})$ when $H$ has at least two dominating vertices.

We obtain these results via a uniform proof of a common technical generalization of both, where we fix the number of $u$-cliques in the graph.  
This general result may be of independent interest.
 
Then we localize these results, proving a tight inequality involving the sizes of the locally largest cliques and complete split graphs. 
}

\section{Introduction}

The generalized Tur\'{a}n number $\ex(n,H,F)$ is the maximum number of subgraphs isomorphic to $H$ in a graph $G$ that has $n$ vertices and does not contain $F$ as a subgraph. More generally, $\ex(n,H,\mathcal{F})$ for a set of graphs $\mathcal{F}$ is the maximum number of subgraphs isomorphic to $H$ in a graph $G$ that has $n$ vertices and contains none of the graphs $F \in \mathcal{F}$ as a subgraph. 

An early result on maximizing the number of copies of a subgraph $H \ne K_2$ is due to Zykov, who in 1949 determined the maximum number of $K_t$'s in a graph not containing any $K_{\w+1}$. We write $k^t(G)$ for the number of subgraphs of $G$ isomorphic to $K_t$.

\begin{thm}[Zykov \cite{Zykov}]\label{thm:zykov}
    For every $n \ge 1$, $\w \ge 1$, and $t \ge 2$, if $G$ is an $n$-vertex, $K_{\w+1}$-free graph, then $\ex(n,K_t,K_{\w+1}) = k^t(\T_{\w}(n))$, where $\T_{\w}(n)$ is the $\w$-partite Tur\'an graph.
\end{thm}

In 2008 Frohmader proved a more detailed form of Zykov's theorem: if we fix the number of edges $m$ instead of the number of vertices $n$, then the maximum number of $K_t$'s in a $K_{\w+1}$-free graph is achieved by $\CT_{\w}(m)$, the $\w$-partite colex Tur\'an graph \cite{Frohmader08}. That is, $\mex(m,K_t,K_{\w+1})=k^t(\CT_{\w}(m))$, where $\mex(m,H,F)$ is the maximum number of copies of $H$ in an $F$-free graph on $m$ edges. The colex Tur\'an graphs interpolate between Tur\'an graphs, and Zykov's theorem follows from Frohmader's and Tur\'an's theorems.

Many other authors have considered generalized Tur\'an problems with $H=K_t$, including when $F$ is a star. Wood proved that $\ex(n,K_t,K_{1,\D+1}) = (1+o_n(1))k^t(\floor{\frac{n}{\D+1}}K_{\D+1})$ \cite{Wood07}. Several years later Chase \cite{Chase20} proved 
the following more precise theorem, 
building on foundational work of Gan, Loh, and Sudakov  
\cite{GLS}. A different proof 
was given by Chao and Dong \cite{CD22}. 
\begin{thm}[Chase \cite{Chase20}]\label{thm:CGLS}
    Let $n$ and $\D$ be positive integers and $t \ge 3$. Then $\ex(n,K_t,K_{1,\D+1}) = k^t(aK_{\D+1}\cup K_b)$, 
    where $a$ and $b$ are defined by $n=a(\D+1)+b$ and $0 \le b < \D+1$.
\end{thm}

Combining the constraints of forbidding large cliques and large stars---i.e., bounding both the clique number and the maximum degree---Radcliffe and the present author in 2020 initiated the study of $\ex(n,K_t,\set{K_{1,\D+1},K_{\w+1}})$ for $t \ge 2$ in \cite{KR20}. We determined upper and lower bounds on the constant $c=c(t,\D,\w)$ such that $\ex(n,K_t,\set{K_{1,\D+1},K_{\w+1}}) = (1-o_n(1))cn$. These bounds on $c$ are equal and achieved by $\T_{\w}(\frac{\D\w}{\w-1})$ when $\D$ is a multiple of $\w-1$, and are asymptotically equivalent as $\D$ increases without bound. The value of $c$ is not always given by the lower bound graph $\T_{\w}(\D+\floor{\frac{\D}{\w-1}})$. In fact we found explicit examples where different graphs maximize the number of $K_t$'s for different values of $t$; e.g., $\ex(42, K_3, \set{K_{1,6},K_5}) = k^3(6\CT_4(17))$, 
while $\ex(42, K_4, \set{K_{1,6},K_5}) = k^4(7\T_4(6))$. 

\cref{table:graphs} shows the extremal graphs for $\ex(n,K_t,\emptyset)$ and $\ex(n,K_t,K_{\w+1})$ and the lower bound graphs for $\ex(n,K_t,K_{1,\D+1})$ and $\ex(n,K_t,\set{K_{1,\D+1},K_{\w+1}})$.

\begin{table}[h]
\centering
\begin{tabular}{c|c|c}
 & Unrestricted & $K_{\w+1}$-free \\\hline
 Unrestricted  & 
    $K_n$ &$\T_{\w}(n)$\\
 	&
 	          \begin{tikzpicture}[scale=.7]
				\draw[fill=gray!20] (2,2) circle [radius=1.8];
			\end{tikzpicture}
 &

            \tikzstyle{vx}=[inner sep=1.5pt,circle,fill=black,draw=black]
			\tikzstyle{edge}=[thick]
			\begin{tikzpicture}[scale=.7]

             \draw[gray!20, line width=9pt] (3,3) -- (1,3);
                    \draw[gray!20, line width=9pt] (3,3) -- (3,1);
                    \draw[gray!20, line width=9pt] (3,3) -- (1,1);
                    \draw[gray!20, line width=9pt] (1,3) -- (3,1);
                    \draw[gray!20, line width=9pt] (1,3) -- (1,1);
                    \draw[gray!20, line width=9pt] (3,1) -- (1,1);
                    
				\draw[rotate around={-45:(1,3)}, fill=white] (1,3) ellipse [x radius=.5, y radius=.9];
				\draw[rotate around={45:(1,1)}, fill=white] (1,1) ellipse [x radius=.5, y radius=.9];
				\draw[rotate around={45:(3,3)}, fill=white] (3,3) ellipse [x radius=.5, y radius=.9];
				\draw[rotate around={-45:(3,1)}, fill=white] (3,1) ellipse [x radius=.5, y radius=.9];
                
			\end{tikzpicture}
 
 \\\hline
 $K_{1,\D+1}$-free & 
 $\floor{\frac{n}{\D+1}}K_{\D+1}$ &  $\floor{\frac{n}{\D+\floor{\D/(\w-1)}}}\T_{\w}(\D+\floor{\frac{\D}{\w-1}})$\\

    &
 	          \begin{tikzpicture}[scale=.4]
				\draw[fill=gray!20] (2,2) circle [radius=1.8];
			\end{tikzpicture}
            \begin{tikzpicture}[scale=.4]
				\draw[fill=gray!20] (2,2) circle [radius=1.8];
			\end{tikzpicture}
            \begin{tikzpicture}[scale=.4]
				\draw[fill=gray!20] (2,2) circle [radius=1.8];
			\end{tikzpicture}
 &

 \tikzstyle{vx}=[inner sep=1.5pt,circle,fill=black,draw=black]
			\tikzstyle{edge}=[thick]
			\begin{tikzpicture}[scale=.4]

             \draw[gray!20, line width=6pt] (3,3) -- (1,3);
                    \draw[gray!20, line width=6pt] (3,3) -- (3,1);
                    \draw[gray!20, line width=6pt] (3,3) -- (1,1);
                    \draw[gray!20, line width=6pt] (1,3) -- (3,1);
                    \draw[gray!20, line width=6pt] (1,3) -- (1,1);
                    \draw[gray!20, line width=6pt] (3,1) -- (1,1);
                    
				\draw[rotate around={-45:(1,3)}, fill=white] (1,3) ellipse [x radius=.5, y radius=.9];
				\draw[rotate around={45:(1,1)}, fill=white] (1,1) ellipse [x radius=.5, y radius=.9];
				\draw[rotate around={45:(3,3)}, fill=white] (3,3) ellipse [x radius=.5, y radius=.9];
				\draw[rotate around={-45:(3,1)}, fill=white] (3,1) ellipse [x radius=.5, y radius=.9];
                
			\end{tikzpicture}\tikzstyle{vx}=[inner sep=1.5pt,circle,fill=black,draw=black]
			\tikzstyle{edge}=[thick]
			\begin{tikzpicture}[scale=.4]

             \draw[gray!20, line width=6pt] (3,3) -- (1,3);
                    \draw[gray!20, line width=6pt] (3,3) -- (3,1);
                    \draw[gray!20, line width=6pt] (3,3) -- (1,1);
                    \draw[gray!20, line width=6pt] (1,3) -- (3,1);
                    \draw[gray!20, line width=6pt] (1,3) -- (1,1);
                    \draw[gray!20, line width=6pt] (3,1) -- (1,1);
                    
				\draw[rotate around={-45:(1,3)}, fill=white] (1,3) ellipse [x radius=.5, y radius=.9];
				\draw[rotate around={45:(1,1)}, fill=white] (1,1) ellipse [x radius=.5, y radius=.9];
				\draw[rotate around={45:(3,3)}, fill=white] (3,3) ellipse [x radius=.5, y radius=.9];
				\draw[rotate around={-45:(3,1)}, fill=white] (3,1) ellipse [x radius=.5, y radius=.9];
                
			\end{tikzpicture}\tikzstyle{vx}=[inner sep=1.5pt,circle,fill=black,draw=black]
			\tikzstyle{edge}=[thick]
			\begin{tikzpicture}[scale=.4]

             \draw[gray!20, line width=6pt] (3,3) -- (1,3);
                    \draw[gray!20, line width=6pt] (3,3) -- (3,1);
                    \draw[gray!20, line width=6pt] (3,3) -- (1,1);
                    \draw[gray!20, line width=6pt] (1,3) -- (3,1);
                    \draw[gray!20, line width=6pt] (1,3) -- (1,1);
                    \draw[gray!20, line width=6pt] (3,1) -- (1,1);
                    
				\draw[rotate around={-45:(1,3)}, fill=white] (1,3) ellipse [x radius=.5, y radius=.9];
				\draw[rotate around={45:(1,1)}, fill=white] (1,1) ellipse [x radius=.5, y radius=.9];
				\draw[rotate around={45:(3,3)}, fill=white] (3,3) ellipse [x radius=.5, y radius=.9];
				\draw[rotate around={-45:(3,1)}, fill=white] (3,1) ellipse [x radius=.5, y radius=.9];
                
			\end{tikzpicture}\\
&& (asymptotically)
\end{tabular}
\captionsetup{width=.8\linewidth}
\caption{Classes of $n$-vertex graphs for which the maximum number of $t$-cliques is known. Gray circles represent complete graphs, and gray bands indicate complete multipartite graphs.
}
\label{table:graphs}
\end{table}

In the last two years, both of the first two problems were extended to maximize the number of copies of a graph $H$ which is not necessarily $K_t$. Morrison, Nir, Norin, {Rz\textogonekcentered{a}$\dot{z}$ewski}, and Wesolek \cite{MNNRW23} proved that for every graph $H$, the Tur\'an graph has the maximum number of subgraphs isomorphic to $H$ among $K_{\w+1}$-free graphs for large enough $\w$. The following theorem is a version\footnote{Note that $\w_0(H)$ is not necessarily equal to what in \cite{MNNRW23} is called $r_0(H)$: for $\w_0(H)$ we require the equation in \cref{thm:turangood} to hold for all $n$, while for $r_0(H)$ it only needs to hold for sufficiently large $n$.} of the impressive main result in \cite{MNNRW23}. We write $\N(H,G)$ for the number of subgraphs isomorphic to $H$ in the graph $G$.

\begin{thm}[Morrison et al. \cite{MNNRW23}]\label{thm:turangood}
	Let $H$ be a graph. There exists $\w_0(H)$ such that, if $\w \ge \w_0(H)$, then for all $n$, $\ex(n,H,K_{\w+1}) = \N(H,\T_{\w}(n))$. Moreover, $\w_0(H) \le 300v(H)^9$.
\end{thm}

In contrast to the exact answer provided in \cref{thm:turangood}, the strongest statement on $\ex(n,H,$ $K_{1,\D+1})$ that is known for all graphs $H$ is that, for all graphs $H$ and trees $T$, $\ex(n,H,T) = \Theta(n^{r(H,T)})$ for an integer $r(H,T)$ defined based on $T$-free partial blowups of $H$ \cite{Letzter19}. 
More precise results on $\ex(n,H,K_{1,\D+1})$ are known for only a few classes of graphs $H$. The case where $H$ is a clique is fully understood (see \cref{thm:CGLS}) and has extremal graphs which are disjoint complete graphs. If $H$ is a star, then all $\D$-regular graphs are extremal. When $H$ is a path, a counting argument shows that the upper bound $\ex(n,P_k,K_{1,\D+1}) \le n\D(\D-1)^{k-2}/2$ is achieved 
by disjoint unions of any $\D$-regular graph $G$ of girth at least $k$ (such graphs $G$ exist for all $\D \ge 3$ and $k\ge 2$ \cite{Sachs63}), so $\ex(n,P_k,K_{1,\D+1})=(1+o_n(1))\N(P_k,\floor[\big]{\frac{n}{\abs{V(G)}}}G)$. If $H$ is a complete bipartite graph $K_{a,b}$ for $2 \le a \le b \le \D$, then 
disjoint unions of balanced complete bipartite graphs are extremal: $\ex(n,K_{a,b},K_{1,\D+1}) = (1+o_n(1))\N(K_{a,b},\floor[\big]{\frac{n}{2\D}}K_{\D,\D})$ \cite[Theorem 1.6(i)]{GP22bipartite}.

In 2024 Nir and the present author \cite{KN24} proved that for every graph $H$ containing a dominating vertex, disjoint complete graphs $K_{\D+1}$ have the maximum number of subgraphs isomorphic to $H$ among $K_{1,\D+1}$-free graphs. In fact, we proved a bit more; since $H\ne K_2$ we may fix the number of edges or number of $u$-cliques for $u > 2$, instead of the number of vertices, to make sure the graph is globally not too large. To state these results, as well as those of the present paper, we introduce the following notation.

\begin{defn}\label{def:exu}
    For positive integers $u$ and $p$, a graph $H$ containing $K_u$ as a subgraph, and a graph $F$ or set of graphs $\mathcal{F}$,
    \begin{align*}
        \ex_u(p,H,F) &:= \max\set{\N(H,G) : \text{$G$ is an $F$-free graph and } k^u(G) = p }, \text{ and}\\
        \ex_u(p,H,\mathcal{F}) &:= \max\set{\N(H,G) : \text{$G$ is an $F$-free graph for every $F\in \mathcal{F}$ and } k^u(G) = p }.
    \end{align*}
\end{defn}
The function $\ex_1(p,H,F)$ is better known as $\ex(p,H,F)$, and $\ex_2(p,H,F)$ is better known as $\mex(p,H,F)$. We require that $H$ contain $K_u$ so that the condition $\N(K_u,G) = p$ ensures that a finite maximum value exists (otherwise, an infinite graph consisting of disjoint copies of $H$ satisfies all constraints). Although this notation is new, $\ex_u(n,H,F)$ has been studied in \cite{AS24,Frohmader08,H77,KR23,KN24}. The number $\mex(m,K_t,F)$ has been determined for $F$ a path \cite{CC24}, star \cite{CC21}, or clique \cite{Frohmader08}; the corresponding numbers $\ex(n,K_t,F)$ are also known for $F$ a path \cite{CC24}, star \cite{Chase20,GLS,CD22}, or clique \cite{turan}. Stability and Erd\H{o}s-Stone type results for $\mex$ were studied in \cite{RU18}. Recently an upper bound on the order of $\mex(m,K_t,F)$ in terms of the order of $\ex(n,K_t,F)$ was found \cite{WXZZ25}.

With this notation, we can state the recent result on graphs with bounded maximum degree.

\begin{thm}[Kirsch and Nir \cite{KN24}]\label{thm:Hstar}
    Let $u$ be a positive integer and $H$ be a graph containing at least $u$ dominating vertices. Then $\ex_u(p,H,K_{1,\D+1}) = (1+o_p(1))\N(H,\floor[\Big]{\frac{p}{\binom{\D+1}{u}}}K_{\D+1})$. In particular, if $H$ is a graph that has a dominating vertex, then $\ex(n,H,K_{1,\D+1}) = (1+o_n(1))\N(H,\floor[\big]{\frac{n}{\D+1}}K_{\D+1})$, and if $H$ has at least two dominating vertices, then $\mex(m,H,K_{1,\D+1})$ $=(1+o_m(1))\N(H,\floor[\Big]{\frac{m}{\binom{\D+1}{2}}}K_{\D+1})$.
\end{thm}

Motivated by this recent progress, the aim of the present paper is to take first steps toward determining $\ex(n,H,\set{K_{1,\D+1},K_{\w+1}})$ for graphs $H\ne K_t$ and $\mex(m,H,\set{K_{1,1,\D+1},K_{\w+1}})$ for any graph $H$. 
We address both questions via $\ex_u(p,H,\set{K_u \vee I_{\D+1},K_{\w+1}})$ for all graphs $H$ containing at least $u$ dominating vertices. Here $K_u\vee I_{\D+1}$ is the complete split graph consisting of a clique of size $u$, an independent set of size $\D+1$, and all edges having one vertex in each. 

The first main theorem, \cref{thm:ex}, determines the asymptotics of $\ex(n,H,\set{K_{1,\D+1},K_{\w+1}})$ for all graphs $H$ containing a dominating vertex. The case $H=K_t$ was proved in \cite{KR20}. \cref{thm:ex} completes a modified form of \cref{table:graphs} where we aim to maximize the number of copies of a graph $H$ that is not necessarily $K_t$, when $K_{1,\D+1}$ is forbidden we require the graph $H$ to have a dominating vertex, and when $K_{\w+1}$ is forbidden we require $\w$ to be sufficiently large.
\begin{thm}\label{thm:ex}
    Let $H$ be a graph containing a dominating vertex. Let $H'$ be the result of deleting a dominating vertex from $H$. Let $n > \D \ge \w \ge w_0(H')+1$ and $L = \T_{\w}(\D+ \floor{\frac{\D}{\w-1}})$. Then
    \[
        \N(H,\floor[\Big]{\frac{n}{\abs{V(L)}}}L) \le \ex(n,H,\set{K_{1,\D+1}, K_{\w+1}}) \le \frac{\N(H',\T_{\w-1}(\D))}{\dom(H)}n,
    \]
and 
\begin{align*}
    \ex(n,H,\set{K_{1,\D+1}, K_{\w+1}}) = (1-o_n(1))c&n, \quad\text{where}\\
    c&=(1+o_{\D}(1))\frac{\N(H,L)}{\abs{V(L)}},
\end{align*}
and $\abs{V(L)}=\D+ \floor{\frac{\D}{\w-1}}$. If $\D$ is a multiple of $\w-1$, then $\abs{V(L)} = \frac{\D\w}{\w-1}$, and
\[
    \ex(n,H,\set{K_{1,\D+1}, K_{\w+1}}) 
    = (1-o_n(1))\frac{\N\parens{H,L}}{\abs{V(L)}}n.
\]
\end{thm}

The second main theorem determines the asymptotics of $\mex(m,H,\set{K_{1,1,\D+1},K_{\w+1}})$ for all graphs $H$ containing at least two dominating vertices. Previously $\mex(m,H,\set{K_{1,1,\D+1},K_{\w+1}})$ had not been considered for any graph $H$, so in particular the result on the asymptotics of $\mex(m,K_t,\set{K_{1,1,\D+1},K_{\w+1}})$ is new.

\begin{thm}\label{thm:mex}Let $H$ be a graph containing at least $2$ dominating vertices. Let $H''$ be the result of deleting $2$ dominating vertices from $H$. Let $m>\D \ge \w \ge w_0(H'')+ 2$ and $L = \T_{\w}(\D+ 2\floor{\frac{\D}{\w-2}})$.
 Then
    \[
    \N(H,\floor[\Big]{\frac{m}{\abs{E(L)}}}L) \le \mex(m,H,\set{K_{1,1,\D+1}, K_{\w+1}}) \le \frac{\N(H'',\T_{\w-2}(\D))}{\binom{\dom(H)}{2}}m,
    \]  
and 
\begin{align*}
    \mex(m,H,\set{K_{1,1,\D+1}, K_{\w+1}}) = (1-o_m(1))c&m,\quad\text{where}\\
    c&= (1+o_{\D}(1))\frac{\N(H,L)}{\abs{E(L)}}.
\end{align*}
If $\D$ is a multiple of $\w-2$, then $L=\T_{\w}(\frac{\D\w}{\w-2})$, and
\[
    \mex(m,H,\set{K_{1,1,\D+1}, K_{\w+1}})
    = (1-o_m(1))\frac{\N(H,L)}{\abs{E(L)}}m.
\]
\end{thm}

Theorems \ref{thm:ex} and \ref{thm:mex} are proved as special cases, $u=1$ and $u=2$ respectively, of the following more general theorem, \cref{thm:summary}. 
The hypotheses that $\w$ is sufficiently large and that $H$ has at least $u$ dominating vertices carry through from Theorems \ref{thm:turangood} and \ref{thm:Hstar}.

\begin{thm}\label{thm:summary}Let $u$ be a positive integer and $H$ be a graph containing at least $u$ dominating vertices. Let $H^{\downarrow u}$ be the result of deleting $u$ dominating vertices from $H$. Let $p>\D \ge \w \ge w_0(H^{\downarrow u})+ u$ and $L = \T_{\w}(\D+ u\floor{\frac{\D}{\w-u}})$.
 Then
    \[
    \N(H,\floor[\Big]{\frac{p}{k^u(L)}}L) \le \ex_u(p,H,\set{K_u\vee I_{\D+1}, K_{\w+1}}) \le \frac{\N(H^{\downarrow u},\T_{\w-u}(\D))}{\binom{\dom(H)}{u}}p,
    \]  
and 
\begin{align*}
    \ex_u(p,H,\set{K_u\vee I_{\D+1}, K_{\w+1}}) = (1-o_p(1))c&p,\quad\text{where}\\
    c&= (1+o_{\D}(1))\frac{\N(H,L)}{k^u(L)}.
\end{align*}
If $\D$ is a multiple of $\w-u$, then $L=\T_{\w}(\frac{\D\w}{\w-u})$, and
\[
    \ex_u(p,H,\set{K_u\vee I_{\D+1}, K_{\w+1}})
    = (1-o_p(1))\frac{\N(H,L)}{k^u(L)}p.
\]
\end{thm}

In \cref{sec:prelim} we set up definitions and notation and prove preliminary results. We prove a lower bound on $\lim_{p\to\infty}\ex_u(p,H,F)/p$ in \cref{sec:lowerbound} and an upper bound in \cref{sec:upperbound}. The bounds are equal and achieved by $\T_{\w}(\frac{\D\w}{\w-u})$ when $\D$ is a multiple of $\w-u$. In \cref{sec:equiv} we show that the bounds are asymptotically equivalent as $\D$ approaches infinity.

In \cref{sec:local} we further generalize the upper bound on $\ex_u(p,H,\set{K_u\vee I_{\D+1}, K_{\w+1}})$ using a new localized approach to generalized Tur\'{a}n problems, introduced in \cite{AS24, Bradac, KN24, MT23}. Instead of globally forbidding large cliques and complete split graphs, we weight each copy of $H$ according to sizes of the locally largest clique and complete split graph and prove a tight inequality. \cref{thm:local} is the first localized form of a generalized Tur\'{a}n theorem that uses multiple distinct weight functions.

We conclude with open problems in \cref{sec:open}.

\section{Preliminaries}\label{sec:prelim}

We write $\cH(G)$ for the set of subgraphs of $G$ that are isomorphic to $H$, so $\N(H,G) = \abs{\cH(G)}$. We also write $\K[s](G)$ for the set of $s$-cliques in $G$, so $k^s(G)=\abs{\K[s](G)} = \N(K_s,G)$. A \emph{dominating vertex} of a graph $H$ is a vertex of $H$ that is adjacent to all other vertices of $H$. We write $\Dom(H)$ for $\set{v \in V(H): v \text{ is a dominating vertex of } H}$ and $\dom(H) := \abs{\Dom(H)}$. Notice that every subset of $\Dom(H)$ is a clique. For a $u$-clique $c$ and a graph $H$ having at least $u$ dominating vertices, $\cH_c(G)$ is the set of subgraphs of $G$ that are isomorphic to $H$ and that have $c \subseteq \Dom(H)$, and $\N_c(H,G):=\abs{\cH_c(G)}$. For a clique $c$, $N(c)$ is the set of common neighbors of the vertices of $c$, i.e., $N(c) = \cap_{v \in c} N(v)$ where $N(v)$ is the set of vertices adjacent to $v$, and $G[N(c)]$ is the induced subgraph of $G$ on $N(c)$.

For a graph $H$ having $\dom(H) \ge u$, we write $H^{\downarrow u}$ for the graph resulting from deleting $u$ dominating vertices of $H$, which is independent of the choice of dominating vertices. If $H$ is an arbitrary graph having at least $u$ dominating vertices, then $H^{\downarrow u}$ is an arbitrary graph, and if $J$ is an arbitrary graph, then $H:=J \vee K_u$ (where $\vee$ is the graph join) has $H^{\downarrow u} = J$.

We are interested in studying $\lim_{p\to\infty}\ex_u(p,H,\mathcal{F})/p$ when it is a finite positive number $c$, i.e. when $\ex_u(p,H,\mathcal{F})=(1-o_p(1))cp$. We also define $\rho_u(H,G) := \N(H,G)/k^u(G)$. While the paper focuses on $\mathcal{F} = \set{K_u\vee I_{\D+1}, K_{\w+1}}$, in this section we prove preliminary results on an arbitrary set of connected graphs $\mathcal{F}$.

\begin{lem}\label{lem:sup}
    For all positive integers $u$, graphs $H$ containing $K_u$, and sets of graphs $\mathcal{F}$ such that every graph $F \in \mathcal{F}$ is connected, 
    \[
        \lim_{p\to\infty} \frac{\ex_u(p,H,\mathcal{F})}{p} = \sup\set[\bigg]{\frac{\ex_u(p,H,\mathcal{F})}{p} : p \in \mathbb{N}} = \sup\set[\Big]{\rho_u(H,G) : G \text{ is } \mathcal{F}\text{-free}}.
    \]
\end{lem}

\begin{proof}
    To prove the first equality, first, we show that $\set{\ex_u(p,H,\mathcal{F})}_{p\ge1}$ is a superadditive sequence. For any positive integers $p_1$ and $p_2$, let $G_1$ and $G_2$ be extremal graphs for $\ex_u(p_1,H,\mathcal{F})$ and $\ex_u(p_2,H,\mathcal{F})$, respectively. By definition $G_1$ and $G_2$ are $\mathcal{F}$-free. Since each graph $F$ in $\mathcal{F}$ is connected, if the disjoint union $G_1 \cup G_2$ contained $F$ then $G_1$ or $G_2$ would contain $F$. Hence $G_1 \cup G_2$ also is $\mathcal{F}$-free, and $k^u(G_1\cup G_2) = p_1 + p_2$. Then 
\[\ex_u(p_1+p_2,H,\mathcal{F}) \ge \N(H,G_1 \cup G_2) \ge \N(H,G_1) + \N(H,G_2) = \ex_u(p_1,H,\mathcal{F}) + \ex_u(p_2,H,\mathcal{F}).\]
    By superadditivity and Fekete's Lemma \cite{Fekete}, the limit exists and equals the supremum.

    The equation $\sup\set{\ex_u(p,H,\mathcal{F})/p : p \in \mathbb{N}}=
    \sup\set{\N(H,G)/p : G \text{ is } \mathcal{F}\text{-free and } k^u(G) = p \text{ for some } p \in \mathbb{N}}$ 
    holds because the former set is a subset of the latter, and each of the values that is in the latter set but not the former is bounded above by some value in the former set. Then $\sup\set{\N(H,G)/p : G \text{ is } \mathcal{F}\text{-free and } k^u(G) = p \text{ for some } p \in \mathbb{N}} 
    = \sup\set{\N(H,G)/k^u(G) : G \text{ is } \mathcal{F}\text{-free}}$.
\end{proof}

In Sections \ref{sec:lowerbound} and \ref{sec:upperbound} we show that $\lim_{p\to\infty}\ex_u(p,H,\set{K_u \vee I_{\D+1},K_{\w+1}})/p = c$ for some $0 < c < \infty$, so $\ex_u(p,H,\set{K_u \vee I_{\D+1},K_{\w+1}})=(1-o_p(1))cp$. The central object of study in this paper is this coefficient $c$, and we call it $P_u(H,\D,\w)$. In this section, we prove preliminary results on its more general form $P_u(H,\mathcal{F})$.

\begin{defn}\label{def:P}
    For all positive integers $u$, graphs $H$ containing $K_u$, and sets of graphs $\mathcal{F}$ such that every graph $F \in \mathcal{F}$ is connected, \[P_u(H,\mathcal{F}):=\lim_{p\to\infty} \frac{\ex_u(p,H,\mathcal{F})}{p}, \text{ and }
    P_u(H,\D,\w):=\lim_{p\to\infty} \frac{\ex_u(p,H,\set{K_u\vee I_{\D+1}, K_{\w+1}})}{p}.\]
\end{defn}

The following proposition is helpful in finding upper and lower bounds on $P_u(H,\mathcal{F})$.

\begin{prop}\label{prop:upperlowerbound}Let $u$ be a positive integer, $H$ be a graph containing $K_u$, and $\mathcal{F}$ be a set of connected graphs. 
For every $\mathcal{F}$-free graph $G$, and every graph $H$,
    $\rho_u(H,G) \le P_u(H,\mathcal{F})$.
    If for some real number $X$ and every $\mathcal{F}$-free graph $G$ we have $\rho_u(H,G) \le X$, then $P_u(H,\mathcal{F})\le X$.
\end{prop}
\begin{proof}By \cref{lem:sup}, $P_u(H,\mathcal{F})$ is the least upper bound of $\set[\big]{\rho_u(H,G) : G \text{ is } \mathcal{F}\text{-free}}$.
\end{proof}

\begin{prop}\label{prop:equal}
    Let $u$ be a positive integer, let $H$ be a graph containing $K_u$ as a subgraph, and let $\mathcal{F}$ be a set of connected graphs that are not contained in $K_u$. If there exists an $\mathcal{F}$-free graph $L$ such that $\rho_u(H,L) = P_u(H,\mathcal{F})$ (i.e., if $P_u(H,\mathcal{F})$ is a maximum and not just a supremum), then
    \[
        \ex_u(p,H,\mathcal{F}) = (1+o_p(1))\N(H,qL\cup rK_u) = (1-o_p(1))\frac{\N(H,L)}{k^u(L)}p,
    \]
    where $q$ and $r$ are defined by $p = q\cdot k^u(L)+r$ and $0 \le r < k^u(L)$. 
    When $k^u(L)$ divides $p$,
    \[\ex_u(p,H,\mathcal{F}) = \N(H,\frac{p}{k^u(L)}L).
    \]
\end{prop}

\begin{proof}
    Suppose the hypotheses of the proposition hold with $\rho_u(H,L) = P_u(H,\mathcal{F})$. Then
    \[
    \lim_{p\to\infty} \ex_u(p,H,\mathcal{F}) / \parens[\bigg]{\frac{\N(H,L)}{k^u(L)}p} = \frac{P_u(H,\mathcal{F})}{\rho_u(H,L)} = 1
    \]
    by definition, and $\ex_u(p,H,\mathcal{F}) = (1-o_p(1))\N(H,L)p/k^u(L)$. 
    
    The graph $qL \cup rK_u$ contains exactly $p$ cliques of size $u$ and is $\mathcal{F}$-free because each graph in $\mathcal{F}$ is connected, and the components of $qL\cup rK_u$ are $\mathcal{F}$-free. Then
    \begin{align*}
        \rho_u(H,L) &= P_u(H,\mathcal{F}) \ge \rho_u(H,qL\cup rK_u) = \N(H,qL\cup rK_u)/p\ge q\N(H,L)/(qk^u(L)+r)\\
        &= \rho_u(H,L)\cdot qk^u(L)/(qk^u(L)+r)
        = \rho_u(H,L)\cdot (1-r/p).
    \end{align*}
    In particular,
    \[
    \rho_u(H,qL\cup rK_u)(1-r/p)\le \rho_u(H,L)(1-r/p) \le \rho_u(H,qL\cup rK_u),
    \]
    and using the fact that $r$ is bounded between constants $0$ and $k^u(L)$ not depending on $p$, 
    \[
        \lim_{p\to\infty}\frac{\rho_u(H,L)}{\rho_u(H,qL\cup rK_u)} = 1.
    \]
    Therefore
    \begin{align*}
        \lim_{p\to\infty}\frac{\ex_u(p,H,\mathcal{F})}{\N(H,qL\cup rK_u)} &= \lim_{p\to\infty}\frac{\ex_u(p,H,\mathcal{F})}{\N(H,qL\cup rK_u)}\cdot \frac{\rho_u(H,qL\cup rK_u)}{\rho_u(H,L)}\\
        &= \lim_{p\to\infty}\frac{\ex_u(p,H,\mathcal{F})}{p\rho_u(H,L)}
        = \frac{P_u(H,\mathcal{F})}{\rho_u(H,L)} = 1,
    \end{align*}
    so $\ex_u(p,H,\mathcal{F}) = (1+o_p(1))\N(H,qL\cup rK_u)$.

    When $k^u(L)$ divides $p$, we have $qL\cup rK_u = (p/k^u(L))L$, and 
    \[
        P_u(H,\mathcal{F})=\rho_u(H,L)= \frac{\N(H,L)}{k^u(L)} = \frac{p}{k^u(L)}\N(H,L)/p = \N(H,\frac{p}{k^u(L)}L)/p = \rho_u\parens[\Big]{H,\frac{p}{k^u(L)}L},
    \]
    and for every $p$ we have $\ex_u(p,H,\mathcal{F})/p \le P_u(H,\mathcal{F})$ by \cref{lem:sup} and \cref{def:P}, so $\N(H,(p/k^u(L))L) \le \ex_u(p,H,\mathcal{F}) \le \rho_u(H,(p/k^u(L))L)p = \N(H,(p/k^u(L))L)$.
\end{proof}

\section{Lower bounds}\label{sec:lowerbound}

In this section we prove lower bounds on both $\ex_u(p,H,\set{K_u \vee I_{\D+1},K_{\w+1}})$ and $P_u(H,\D,\w) = \lim_{p\to\infty} \ex_u(p,H,\set{K_u \vee I_{\D+1},K_{\w+1}})/p$. To do so we define a lower bound graph.

Let $\D \ge \w \ge u+1 \ge 2$, and define $a$ and $b$ by $\D = a(\w-u)+b$ and $0 \le b \le \w-u-1$. Notice that 
\begin{equation}\label{eq:divisionedgesclique}
    \D+u\floor[\bigg]{\frac{\D}{\w-u}} = \D + ua = (a(\w-u)+b)+ua = a\w+b.
\end{equation}

\begin{defn}\label{def:LBgraphs}
    The \emph{lower bound graph} $L_u(\D,\w)$, called $L$ when $u$, $\D$, and $\w$ are clear, is $\T_{\w}(\D+ u\floor{\frac{\D}{\w-u}}) = \T_{\w}(a\w+b)$.
\end{defn}

\begin{lem}\label{lem:free}
    For every $\D \ge \w \ge u+1 \ge 2$ and every $p \ge 1$, define $q$ and $r$ by $p=q\cdot k^u(L)+r$ and $0 \le r < k^u(L)$. Then $L$ and $qL\cup rK_u$ are $\set{K_u \vee I_{\D+1},K_{\w+1}}$-free, and $k^u(qL\cup rK_u) = p$.
\end{lem}

\begin{proof}
    The graph $L=\T_{\w}(a\w+b)$ is $\w$-partite so $K_{w+1}$-free. If $c$ 
    is a $u$-clique in this graph, then $N(c)$ is the set of vertices in the $\w-u$ parts containing no vertices of $c$. By \cref{eq:divisionedgesclique}, $N(c)$ contains at most $(\D+ua) - ua = \D$ vertices. As $c$ was arbitrary, the graph $L$ is $K_u \vee I_{\D+1}$-free. The graph $qL\cup rK_u$ also is $\set{K_u \vee I_{\D+1},K_{\w+1}}$-free because its components are. Finally, $k^u(qL\cup rK_u) = q\cdot k^u(L)+r = p$.
\end{proof}

\cref{prop:lowerboundrhou} gives the lower bound on $P_u(H,\D,\w)$ used to determine $P_u(H,\D,\w)$ asymptotically in \cref{thm:rhouasymp}.
\begin{prop}\label{prop:lowerboundrhou}
    For every $\D \ge \w \ge u+1 \ge 2$ and every graph $H$ containing $K_u$ as a subgraph,  $\rho_u(H,L) \le P_u(H,\D,\w)$. That is, the lower bound graph gives a lower bound on $P_u(H,\D,\w)$.
\end{prop}

\begin{proof}
    By \cref{lem:free}, $L$ is $\set{K_u \vee I_{\D+1},K_{\w+1}}$-free. Every $\set{K_u \vee I_{\D+1},K_{\w+1}}$-free graph $G$ by \cref{prop:upperlowerbound} satisfies $\rho_u(H,G) \le P_u(H,\D,\w)$.
\end{proof}

\cref{prop:cliqueLB} gives the lower bound on $\ex_u(p,H,\set{K_u \vee I_{\D+1},K_{\w+1}})$ in \cref{thm:summary}.

\begin{prop}\label{prop:cliqueLB}
    For every $\D \ge \w \ge u+1 \ge 2$, every $p\ge 1$,  and every graph $H$ containing $K_u$ as a subgraph,  $\N(H,\floor[\big]{\frac{p}{k^u(L)}}L) \le \ex_u(p,H,\set{K_u \vee I_{\D+1},K_{\w+1}})$.
\end{prop}

\begin{proof}
    By \cref{lem:free} and \cref{def:exu}, and defining $q=\floor{p/k^u(L)}$ and $r$ as in \cref{lem:free},
    \[
    \N(H,\floor[\Big]{\frac{p}{k^u(L)}}L) \le \N(H,\floor[\Big]{\frac{p}{k^u(L)}}L\cup rK_u) \le \ex_u(p,H,\set{K_u \vee I_{\D+1},K_{\w+1}}).\qedhere
    \]
\end{proof}

\section{Upper bounds}\label{sec:upperbound}

In this section we prove the upper bound on the generalized Tur\'an number $\ex_u(p,H,\set{K_u\vee I_{\D+1}, K_{\w+1}})$ stated in \cref{thm:summary}, and the upper bound on $P_u(H,\D,\w)$ used in \cref{thm:rhouasymp}. We show that these upper bounds are achieved in the special case that $\w-u$ divides $\D$.

\begin{thm}\label{thm:rhou}
    For an integer $u \ge 1$, let $H$ be a graph having $\dom(H) \ge u$. For $\D \ge \w \ge \w_0(H^{\downarrow u})+u$, 
    \[
        P_u(H,\D,\w) \le \frac{\N(H^{\downarrow u},\T_{\w-u}(\D))}{\binom{\dom(H)}{u}} \quad\text{and} \quad \ex_u(p,H,\set{K_u\vee I_{\D+1}, K_{\w+1}}) \le \frac{\N(H^{\downarrow u},\T_{\w-u}(\D))}{\binom{\dom(H)}{u}} p.
    \]
    If $\D$ is a multiple of $\w-u$, then this upper bound is achieved by the lower bound graph $L=\T_{\w}(\frac{\D\w}{\w-u})$, i.e., 
    \[
        P_u(H,\D,\w) = \frac{\N(H^{\downarrow u},\T_{\w-u}(\D))}{\binom{\dom(H)}{u}} = \rho_u(H,\T_\w\parens[\Big]{\frac{\w\D}{\w-u}}).
    \]
\end{thm}

\begin{proof}
    Let $G$ be a $\set{K_u \vee I_{\D+1},K_{\w+1}}$-free graph. Let $c$ be a $u$-clique of $G$. Then $G[N(c)]$ is a $K_{\w-u+1}$-free graph on at most $\D$ vertices, and $\w-u \ge \w_0(H^{\downarrow u})$, so by \cref{thm:turangood}, $\N(H^{\downarrow u},G[N(c)]) \le \N(H^{\downarrow u}, \T_{\w-u}(\D))$. By counting in two ways the number of ordered pairs $(c, J)$ where $J$ is a subgraph of $G$ isomorphic to $H$ that contains every vertex of $c$ as a dominating vertex (so $c \in \K[u](G)$), we have
	\begin{align*}
		\binom{\dom(H)}{u}\cdot \N(H,G) &= \sum_{c \in \K[u](G)}\N_{c}(H,G)= \sum_{c \in \K[u](G)} \N(H^{\downarrow u},G[N(c)])\\
		& \le \sum_{c \in \K[u](G)} \N(H^{\downarrow u}, \T_{\w-u}(\D)) = k^u(G) \cdot \N(H^{\downarrow u}, \T_{\w-u}(\D)),
	\end{align*}
	so
	\begin{equation}
	\rho_u(H,G) = \frac{\N(H,G)}{k^u(G)} \le \frac{\N(H^{\downarrow u},\T_{\w-u}(\D))}{\binom{\dom(H)}{u}}\label{eq:UB}
	\end{equation}
    for every $\set{K_u \vee I_{\D+1}, K_{\w+1}}$-free graph $G$, and by \cref{prop:upperlowerbound},
    \[
        P_u(H,\D,\w) \le \frac{\N(H^{\downarrow u},\T_{\w-u}(\D))}{\binom{\dom(H)}{u}}.
    \]
    From \cref{eq:UB} we also have $\N(H,G) \le \N(H^{\downarrow u},\T_{\w-u}(\D))k^u(G)/\binom{\dom(H)}{u}$ for every $\set{K_u \vee I_{\D+1}, K_{\w+1}}$-free graph $G$, so $\ex_u(p,H,\set{K_u\vee I_{\D+1}, K_{\w+1}}) \le \N(H^{\downarrow u},\T_{\w-u}(\D))p/\binom{\dom(H)}{u}$.
    
	When $w-u$ divides $\D$, the lower bound graph $L$ is $\T_{\w}(\D+u\floor{\frac{\D}{\w-u}}) = \T_{\w}(\frac{\w\D}{\w-u})$ (\cref{def:LBgraphs}). Each of the $\w$ parts has size exactly $\frac{\D}{w-u}$, so each $u$-clique $c$ has $G[N(c)] \cong \T_{\w-u}(\D)$. Therefore $L$ achieves equality in the above bound and is $\set{K_u \vee I_{\D+1},K_{\w+1}}$-free by \cref{lem:free}, so $P_u(H,\D,\w) = \rho_u(H,L) = \N(H^{\downarrow u},\T_{\w-u}(\D))/\binom{\dom(H)}{u}$.
    \end{proof}

We also draw the following conclusions about the generalized Tur\'an number when $w-u$ divides $\D$.

\begin{cor}
    For an integer $u\ge 1$, let $H$ be a graph having $\dom(H) \ge u$. Let $\D \ge \w \ge \w_0(H^{\downarrow u})+u$, where $\w-u$ divides $\D$. Then $L=\T_{\w}\parens{\frac{\w\D}{\w-u}}$,
    \begin{align*}
        \ex_u(p,H,\set{K_u\vee I_{\D+1},K_{\w+1}}) &= (1+o_p(1))\N(H,qL\cup rK_u)\\
        &= (1-o_p(1))\frac{\N(H,L)}{k^u(L)}p,
    \end{align*}
    where $q$ and $r$ are defined by $p = q\cdot k^u(L) + r$ and $0 \le r < k^u(L)$,
    and for $p$ a multiple of $k^u(L)$,
    \[
    \ex_u(p,H,\set{K_u\vee I_{\D+1},K_{\w+1}}) = \N(H,\frac{p}{k^u(L)}L).
    \]
\end{cor}

\begin{proof}
    By \cref{thm:rhou}, $P_u(H,\D,\w) = \rho_u(H,\T_{\w}(\frac{\w\D}{w-u}))$. Apply \cref{prop:equal}.
\end{proof}

\section{Asymptotic equivalence of upper and lower bounds}\label{sec:equiv}
When $\w-u$ does not divide $\D$, the lower and upper bounds on $P_u(H,\D,\w)$ given by \cref{prop:lowerboundrhou} and \cref{thm:rhou} are not equal, but in \cref{thm:rhouasymp} we show that they are asymptotically equivalent. 
\cref{cor:turancount} and \cref{lem:ulimit} on subgraph counts in Tur\'an graphs play a key role in that proof. 
We state \cref{prop:turancount} for all values of $u$ although in fact we use only the case $u=1$ (\cref{cor:turancount}) in the proofs of \cref{lem:ulimit} and \cref{thm:rhouasymp}.

\begin{lem}\label{prop:turancount}
    Let $H$ be a graph, $r \ge \w_0(H)$, $1 \le u \le r$, and $n \ge 1$. Let $d$ be a $u$-clique in $\T_r(n)$ whose vertices lie in $u$ largest parts of $\T_r(n)$. Let $c$ be an arbitrary $u$-clique in $\T_r(n)$. Then the number of copies of $H$ in $\T_r(n)$ containing at least one vertex of $c$ is at least the number containing at least one vertex of $d$, which is $
    \N(H,\T_r(n)) - \N(H,\T_r(n-u))$.
\end{lem}

\begin{proof}
    Let $c$ be an arbitrary $u$-clique in $\T_r(n)$. The set of copies of $H$ in $\T_r(n)$ can be partitioned into those that contain no vertices of $c$ and those that contain at least one vertex of $c$. The former set has size $\N(H,K)$, where $K$ is the graph resulting from deleting $c$ from $\T_r(n)$. Since $K$ is $K_{r+1}$-free on $n-u$ vertices, and $r \ge \w_0(H)$, by \cref{thm:turangood} we have $\N(H,K) \le \N(H,\T_r(n-u))$. Therefore the number of copies of $H$ in $\T_r(n)$ containing at least one vertex of $c$ is $\N(H,\T_r(n))-\N(H,K) \ge \N(H,\T_r(n))-\N(H,\T_r(n-u))$, and equality is achieved when the vertices of $c$ lie in $u$ largest parts of $\T_r(n)$.
\end{proof}

The following corollary is the $u=1$ case of \cref{prop:turancount}.
\begin{cor}\label{cor:turancount}
    Let $H$ be a graph, $r \ge \w_0(H)$, and $n \ge 1$. Let $x$ be a vertex in a larger part of $\T_r(n)$, and let $v$ be a vertex in a smaller part. Then the number of copies of $H$ in $\T_r(n)$ containing $v$ is at least the number containing $x$, which is $\N(H,\T_r(n)) - \N(H,\T_r(n-1))$.
\end{cor}

Now we use \cref{cor:turancount} to determine a limit that we need for \cref{thm:rhouasymp}.

\begin{lem}\label{lem:ulimit}
    Let $H$ be a graph, $u\ge 1$, and $r \ge \w_0(H)$. 
    Then 
    \[
        \lim_{n\to\infty}\frac{\N(H,\T_{r}(n-u))}{\N(H,\T_{r}(n))} = 1.
    \]
\end{lem}

\begin{proof}
    We address the $u=1$ case first. 
    Let $S = \set{(v,J) : J \in \cH(\T_r(n)), v \in V(J)}$. Counting $\abs{S}$ by choosing the vertex first, by \cref{cor:turancount} we have
        $\abs{S} \ge n(\N(H,\T_r(n))-\N(H,\T_r(n-1)))$. 
    Counting $\abs{S}$ by choosing the copy $J$ of $H$ first, $\abs{S}= \N(H,\T_r(n))\abs{V(H)}$. 
    Therefore
    \begin{align*}
        \N(H,\T_r(n))\abs{V(H)} &\ge n(\N(H,\T_r(n))-\N(H,\T_r(n-1)))\\
        \frac{\abs{V(H)}}{n} &\ge 1 - \frac{\N(H,\T_r(n-1))}{\N(H,\T_r(n))}\\
        1 \ge \frac{\N(H,\T_r(n-1))}{\N(H,\T_r(n))} &\ge 1 - \frac{\abs{V(H)}}{n}.
    \end{align*}
    As $\abs{V(H)}$ does not depend on $n$, we have $\displaystyle \lim_{n\to\infty}\frac{\N(H,\T_r(n-1))}{\N(H,\T_r(n))} = 1$.
    Then, for any positive integer $u$,
    \[
    \lim_{n\to\infty} \frac{\N(H,\T_r(n-u))}{\N(H,\T_r(n))}
        =\lim_{n\to\infty} \prod_{i=0}^{u-1} \frac{\N(H,\T_r(n-i-1))}{\N(H,\T_r(n-i))}
        = \prod_{i=0}^{u-1} \lim_{n\to\infty}\frac{\N(H,\T_r(n-i-1))}{\N(H,\T_r(n-i))}
        = 1,
    \]
    where the first equality is a telescoping product.
\end{proof}

\cref{thm:rhouasymp} determines $P_u(H,\D,\w)$ asymptotically.

\begin{thm}\label{thm:rhouasymp}
Let $u\ge 1$ be an integer, let $H$ be a graph with $\dom(H) \ge u$, and let $\w \ge \w_0(H^{\downarrow u})+u$. Then 
\[
P_u(H, \D, \w) = (1+o_{\D}(1))\rho_u(H,L).
\]
\end{thm}

\begin{proof}
    By \cref{prop:lowerboundrhou} and \cref{thm:rhou}, 
    \[
    \frac{\N(H,L)}{k^u(L)}=\rho_u(H,L) \le P_u(H,\D,\w) \le \frac{\N(H^{\downarrow u},\T_{\w-u}(\D))}{\binom{\dom(H)}{u}}.
    \]
    It is enough to show that the limit as $\D$ increases without bound of the ratio of the lower bound to the upper bound is $1$.

    Let $c$ be a $u$-clique in $L = \T_{\w}(\D+ua)$ by \cref{eq:divisionedgesclique}. Then $L[N(c)] = \T_{\w-u}(\D+ua-x)$, where $x$ is the total number of vertices in the $u$ parts containing a vertex of $c$, so $x \le (a+1)u$, and $L[N(c)] \supseteq T_{\w-u}(\D-u)$. Therefore,
    \begin{align*}
        \binom{\dom(H)}{u}\cdot \N(H,L) &= \sum_{c \in \K[u](L)}\N_c(H,L)\\
        &= \sum_{c \in \K[u](L)} \N(H^{\downarrow u},L[N(c)])\\
        &\ge \sum_{c \in \K[u](L)} \N(H^{\downarrow u},\T_{w-u}(\D-u))= k^u(L)\N(H^{\downarrow u},\T_{w-u}(\D-u)),
    \end{align*}
    so \[\rho_u(H,L) = \frac{\N(H,L)}{k^u(L)} \ge \frac{\N(H^{\downarrow u},\T_{w-u}(\D-u))}{\binom{\dom(H)}{u}}.\] Then the limit of the ratio of the lower bound to the upper bound is 
    \[
        1 \ge \lim_{\D\to\infty}\frac{\N(H,L)\binom{\dom(H)}{u}}{k^u(L)\N(H^{\downarrow u},\T_{\w-u}(\D))} 
        \ge \lim_{\D\to\infty} \frac{\N(H^{\downarrow u},\T_{w-u}(\D-u))\binom{\dom(H)}{u}}{\binom{\dom(H)}{u}\N(H^{\downarrow u},\T_{\w-u}(\D))}=1
    \]
    by \cref{lem:ulimit} (substituting $H^{\downarrow u}$ for $H$, $\w-u$ for $r$, and $\D$ for $n$).
\end{proof}

\section{Localization}\label{sec:local}

Localization of extremal combinatorics theorems was introduced in \cite{Bradac, MT23} and extended to generalized Tur\'{a}n problems in \cite{AS24,KN24}. 
The tight upper bound on $\ex_u(p,H,\set{K_u \vee I_{\D+1}, K_{\w+1}})$ from \cref{thm:rhou} can be localized as follows to obtain a tight inequality involving two weight functions, $\w(J)$ and $\D(J)$, giving the sizes of the locally largest cliques and complete split graphs containing a set of $u$ dominating vertices of each copy $J$ of $H$.

\begin{defn}
Let $u \ge 1$ be an integer, and let $H$ be a graph with $\dom(H) \ge u$. 

For a graph $G$ and a $u$-clique $c$ of $G$, let $\w(c)$ be the size of a largest clique of $G$ containing $c$, i.e., $\w(c)=\max\set{|S| : c \subseteq S \subseteq V(G) \text{ and } G[S]\cong K_{|S|}}$. For $J \in \cH(G)$, we define $\w(J) = \max\set{\w(c) : c \in \binom{\Dom(J)}{u}}$, the maximum value of $\w(c)$ over all $u$-cliques $c$ consisting of dominating vertices of $J$.

We define $\D(c)$ as the number of common neighbors of $c$ in $G$, i.e., $\D(c) = \abs{N(c)} = \abs{\cap_{v \in c}N(v)}$, and $\D(J) = \max\set{\D(c): c \in \binom{\Dom(J)}{u}}$, the maximum value of $\D(c)$ over all $u$-cliques $c$ consisting of dominating vertices of $J$.
\end{defn}

Notice that $\w(c) \ge u$, so $\w(c)$ and $\w(J)$ are well-defined. The definition of $\w(J)$ is necessarily different from the definition of $\alpha_G(T)$ given in \cite[Theorem 4]{KN24} (the size of a largest clique containing the clique $T$) because $H$ may not be a clique and hence $J$ may not be contained in a clique of any size in $G$.

\begin{prop}\label{prop:local}
    Let $u \ge 1$ be an integer, and let $H$ be a graph with $\dom(H) \ge u$. In a graph $G$, for every pair $(c,J)$, where $J \in \cH(G)$ and $c \in \binom{\Dom(J)}{u}$, we have 
    \begin{enumerate}
        \item $\w(c) \le \w(J)$,
        \item $\D(c) \le \D(J)$,
        \item $\D(c) \ge \w(c)-u$, and
        \item $\D(J) \ge \w(J)-u$.
    \end{enumerate}
\end{prop}

\begin{proof}
    The first two statements hold because by definition $\w(J)$ and $\D(J)$ are the maximums over sets containing $\w(c)$ and $\D(c)$, respectively. For each $u$-clique $c$, we have $\D(c) \ge \w(c)-u$ because every vertex in a clique containing $c$ is in the common neighborhood of $c$, so also $\D(J) = \max\set{\D(c): c \in \binom{\Dom(J)}{u}} \ge \max\set{\w(c)-u: c \in \binom{\Dom(J)}{u}} = \w(J)-u$ for every $J \in \cH(G)$.
\end{proof}

The following proposition is used in the proof of \cref{thm:local}.
\begin{prop}\label{prop:turancounts}
    Let $H$ be a graph and $n$, $r$, and $s$ be positive integers. If $r \ge w_0(H)$ and $r \ge s$ then $\N(H,\T_r(n)) \ge \N(H,\T_s(n))$.
\end{prop}
\begin{proof}
    By \cref{thm:turangood}, $\N(H,\T_r(n)) = \ex(n,H,K_{r+1})$, the maximum value of $\N(H,G)$ over all $n$-vertex, $K_{r+1}$-free graphs $G$, and $\T_s(n)$ is an $n$-vertex, $K_{r+1}$-free graph.
\end{proof}

\cref{thm:local} is a localized form of \cref{thm:rhou}.

\begin{thm}\label{thm:local}
    Let $u \ge 1$ be an integer, and let $H$ be a graph with $\dom(H) \ge u$. Let $G$ be a graph in which every $u$-clique $c$ consisting of dominating vertices of a copy of $H$ has $\w(c) \ge \w_0(H^{\downarrow u})+u$. Then $x(J):=\displaystyle\frac{1}{\N(H^{\downarrow u}, \T_{\w(J)-u}(\D(J)))}$ is well-defined and weakly decreasing in both $\w(J)$ and $\D(J)$, and
    \[\displaystyle\sum_{J \in \cH(G)}x(J) \le \frac{k^u(G)}{\binom{\dom(H)}{u}}.\]
    Equality holds for any disjoint union of balanced Tur\'an graphs, $G = \T_{\w_1}(a_1\w_1) \cup \cdots \cup \T_{\w_k}(a_k\w_k)$, with $\w_i \ge \w_0(H^{\downarrow u})+u$ for all $1 \le i \le k$, as well as for $G = \T_{\w_1}(a_1\w_1) \cup \cdots \cup \T_{\w_k}(a_k\w_k)\cup Z$, where $Z$ is any $K_u$-free graph.
\end{thm}

\begin{proof} 
    Every $J \in \cH(G)$, since $J$ is a copy of $K_u \vee H^{\downarrow u}$ and contains at least one $u$-clique $c$ of dominating vertices of $J$, by part 1 of \cref{prop:local} has $\w(J) \ge \w(c) \ge \w_0(H^{\downarrow u})+u$.
    
    Let $J \in \cH(G)$, and let $c \in \binom{\Dom(J)}{u}$ such that $\w(J) = \w(c)$. By part 2 of \cref{prop:local}, $\D(J) \ge \D(c) \ge \abs{V(H)}-u$, using the fact that $c \subseteq \Dom(J)$. Consider the graph $\widehat{H} := H^{\downarrow u} \cup (\D(J) - \abs{V(H)}+u)K_1$, which is well-defined by the previous sentence and has $\D(J)$ vertices. The graph $\widehat{H}$ also is $K_{\w(J)-u+1}$-free because otherwise, $H^{\downarrow u}$ would contain a $K_{\w(J)-u+1}$, and $c$ would be contained in a $K_{\w(J)+1}=K_{\w(c)+1}$, contradicting the definition of $\w(c)$. Therefore there exists a $\D(J)$-vertex, $K_{\w(J)-u+1}$-free graph, namely $\widehat{H}$, which contains a copy of $H^{\downarrow u}$. Using the fact that $\w(J) \ge \w_0(H^{\downarrow u})+u$ from the previous paragraph and \cref{thm:turangood}, $\N(H^{\downarrow u}, \T_{\w(J)-u}(\D(J))) \ge \N(H^{\downarrow u}, \widehat{H}) \ge 1$. Therefore $x(J)$ is well-defined.

    The function $x(J)$ is weakly decreasing in $\D(J)$ because $\T_{r}(n) \subseteq \T_{r}(n+1)$ for all $n$ and $r$. It is weakly decreasing in $\w(J)$ by \cref{prop:turancounts}, using the fact that $\w(J)-u \ge \w_0(H^{\downarrow u})$ from the first paragraph.
    
    Also note that, for a $u$-clique $c$ in $G$, the induced subgraph on its common neighborhood, $G[N(c)]$, by definition has $\D(c)$ vertices and maximum clique size $\w(c)-u$. For $c \in \binom{\Dom(J)}{u}$ for some $J \in \cH(G)$, we have $\w(c)-u \ge \w_0(H^{\downarrow u})$, so by \cref{thm:turangood}, we have $\N(H^{\downarrow u},G[N(c)]) \le \N(H^{\downarrow u},T_{\w(c)-u}(\D(c)))$. Recall that for a $u$-clique $c$, we write $\cH_c(G)$ for the set of copies $J$ of $H$ in $G$ such that $c \in \binom{\Dom(J)}{u}$. If a $u$-clique $c$ is not a set of dominating vertices of some copy of $H$, then $c$ does not contribute to the sum of $x(J)$'s, and $\N(H^{\downarrow u},G[N(c)]) = \abs{\cH_c(G)} = 0 \le \N(H^{\downarrow u},\T_{\w(c)-u}(\D(c)))$. Therefore for all $c \in \K[u](G)$,
    \begin{equation}\label{eq:local}
        \N(H^{\downarrow u},G[N(c)]) \le \N(H^{\downarrow u},\T_{\w(c)-u}(\D(c))).
    \end{equation}

    Summing $x(J)$ over pairs $(c, J)$ such that $J \in \cH(G)$ and $c \in \binom{\Dom(J)}{u}$ in two ways, we have
    \begin{align*}
        &\binom{\dom(H)}{u}\sum_{J \in \cH(G)}\frac{1}{\N(H^{\downarrow u}, \T_{\w(J)-u}(\D(J)))} \\
        &= \sum_{c \in \K[u](G)} \sum_{J \in \cH_c(G)} \frac{1}{\N(H^{\downarrow u}, \T_{\w(J)-u}(\D(J)))} \\
        &\le \sum_{c \in \K[u](G)} \sum_{J \in \cH_c(G)} \frac{1}{\N(H^{\downarrow u}, \T_{\w(c)-u}(\D(J)))} \quad\text{by Propositions \ref{prop:local} and \ref{prop:turancounts}, and } \w(J)-u \ge \w_0(H^{\downarrow u})\\
        &\le \sum_{c \in \K[u](G)} \sum_{J \in \cH_c(G)} \frac{1}{\N(H^{\downarrow u}, \T_{\w(c)-u}(\D(c)))} \quad\text{by \cref{prop:local}}\\
        &= \sum_{c \in \K[u](G)} \frac{\abs{\cH_c(G)}}{\N(H^{\downarrow u}, \T_{\w(c)-u}(\D(c)))} = \sum_{c \in \K[u](G)} \frac{\N(H^{\downarrow u},G[N(c)])}{\N(H^{\downarrow u}, \T_{\w(c)-u}(\D(c)))} \\
        &\le \sum_{c \in \K[u](G)} 
        \frac{\N(H^{\downarrow u},\T_{\w(c)-u}(\D(c)))}
        {\N(H^{\downarrow u}, \T_{\w(c)-u}(\D(c)))}
        = \sum_{c \in \K[u](G)} 1 = k^u(G) \quad\text{by \cref{eq:local}}.
    \end{align*}
    Suppose $G = \T_{\w_1}(a_1\w_1) \cup \cdots \cup \T_{\w_k}(a_k\w_k)\cup Z$, as in the theorem statement, where $Z$ may be empty. Let $J$ be a copy of $H$ in $G$. Then $J$ is in the $\T_{\w_i}(a_j\w_i)$ component for some $1 \le i \le k$. Every $u$-clique $c$ of dominating vertices of $J$ has the same value of $\w(c) = \w_i$, so $\w(c) = \w(J)$. Also, every $u$-clique $c$ of dominating vertices of $J$ has the same value of $\D(c) = (\w_i-u)a_i$, so $\D(c) = \D(J)$. Finally, every $u$-clique $c$ in $G$ has $G[N(c)] \cong \T_{\w(c)-u}(\D(c))$, using that $Z$ is $K_u$-free. Therefore the three inequalities in the above proof all are equalities, and for this graph
    \[
    \sum_{J \in \cH(G)}\frac{1}{\N(H^{\downarrow u}, \T_{\w(J)-u}(\D(J)))} = \frac{k^u(G)}{\binom{\dom(H)}{u}}.\qedhere
    \]
\end{proof}

We can recover the upper bound $\N(H,G) \le \N(H^{\downarrow u},\T_{\w-u}(\D))p/\binom{\dom(H)}{u}$ from \cref{thm:local} by adding global hypotheses on the sizes of the largest cliques and complete split graphs, but the hypothesis on $\w(c)$ is still needed to use \cref{thm:local}:

\begin{cor}
    Let $u \ge 1$ be an integer, and let $H$ be a graph with $\dom(H) \ge u$. Let $\D \ge \w \ge \w_0(H^{\downarrow u})+u$. Let $G$ be a $\set{K_u\vee I_{\D+1}, K_{\w+1}}$-free graph in which every $u$-clique $c$ consisting of dominating vertices of a copy of $H$ has $\w(c) \ge \w_0(H^{\downarrow u})+u$. Then 
    \[
    \N(H,G) \le \frac{\N(H^{\downarrow u}, \T_{\w-u}(\D))k^u(G)}{\binom{\dom(H)}{u}}.
    \]
\end{cor}

\begin{proof}
    Every $u$-clique $c$ in $G$ has $\w(c) \le \w$ and $\D(c) \le \D$, so every $J \in \cH(G)$ has $\w(J) \le \w$ and $\D(J)\le \D$. 
    By \cref{thm:local}, $x(J)$ is weakly decreasing in both $\w(J)$ and $\D(J)$, so for all $J$ we have $x(J) = 1/\N(H^{\downarrow u}, \T_{\w(J)-u}(\D(J))) \ge 1/\N(H^{\downarrow u}, \T_{\w-u}(\D))$. Then by \cref{thm:local},
    \[
    \frac{\N(H,G)}{\N(H^{\downarrow u}, \T_{\w-u}(\D))}=\sum_{J \in \cH(G)}\frac{1}{\N(H^{\downarrow u}, \T_{\w-u}(\D))} \le \sum_{J \in \cH(G)}x(J) \le \frac{k^u(G)}{\binom{\dom(H)}{u}}.\qedhere
    \]
\end{proof}

In the special case where $H$ is a clique, we can remove the $\w(c)$ hypothesis, and we can use a general formula for the number of copies of $H^{\downarrow u}$ in the Tur\'an graph. 

\begin{cor}
    Let $t \ge u+1 \ge 2$. 
    Every graph $G$ satisfies
    \[
    \sum_{J \in \K[t](G)}\frac{1}{k^{t-u}(\T_{\w(J)-u}(\D(J)))} \le \frac{k^u(G)}{\binom{t}{u}},
    \]
    i.e., defining $a=a(J)$ and $b=b(J)$ by $\D(J) = a(\w(J)-u)+b$ and $0 \le b < \w(J)-u$,
    \[
    \sum_{J \in \K[t](G)}\frac{1}{\sum_{i=0}^b\binom{b}{i}\binom{\w(J)-u-b}{t-u-i} (a+1)^i a^{t-u-i}} \le \frac{k^u(G)}{\binom{t}{u}},
    \]
    Equality holds for any disjoint union of balanced Tur\'an graphs, $G = \T_{\w_1}(a_1\w_1) \cup \cdots \cup \T_{\w_k}(a_k\w_k)$, 
    as well as for $G = \T_{\w_1}(a_1\w_1) \cup \cdots \cup \T_{\w_k}(a_k\w_k)\cup Z$, where $Z$ is any $K_u$-free graph.
\end{cor}

\begin{proof}
    Notice that by \cref{thm:zykov}, $\w_0(K_{t-u}) = 1$. Every $u$-clique $c$ contained in a copy of $K_t$ has $\w(c) \ge t \ge u+1 = \w_0(K_{t-u})+u$, so \cref{thm:local} applies with $H = K_t$, $\dom(H) = t$, $H^{\downarrow u} = K_{t-u}$, and \[
    \N(H^{\downarrow u}, \T_{\w(J)-u}(\D(J))) = k^{t-u}(\T_{\w(J)-u}(\D(J))) =  \sum_{i=0}^b\binom{b}{i}\binom{\w(J)-u-b}{t-u-i} (a+1)^i a^{t-u-i},\]
where $\D(J) = a(\w(J)-u)+b$ and $0 \le b < \w(J)-u$ \cite[Lemma 3.1]{KR20}. 
\end{proof}

\section{Open Problems}\label{sec:open}

These results suggest the following open problems. First, we determined $\ex_u(p,H,\set{K_u\vee I_{\D+1},$ $K_{\w+1}})$ asymptotically when $\D$ is a multiple of $\w-u$, but the exact values are known only when $n$ is a multiple of $\frac{\D\w}{\w-u}$, i.e. when the extremal graphs are disjoint copies of $\T_{\w}(\frac{\D\w}{\w-u})$.

\begin{problem}\label{problem:1}
    For $H$ a graph containing $u$ dominating vertices, $\D$ a multiple of $\w-u$, and $n$ not a multiple of $\frac{\D\w}{\w-u}$, determine the exact value of $\ex_u(p,H,\set{K_u\vee I_{\D+1},K_{\w+1}})$.
\end{problem}

For \cref{problem:1} it would be reasonable to consider, for $u=1$, the graph $a\T_{\w}(\frac{\D\w}{\w-1})\cup \T_{\w}(b)$, where $a$ and $b$ are defined by $n = a\frac{\D\w}{\w-1} + b$ and $0 \le b < \frac{\D\w}{\w-1}$, and for $u=2$, the graph $a\T_{\w}(\frac{\D\w}{\w-2})\cup \CT_{\w}(b)$, where $a$ and $b$ are defined by $m = a\cdot e(\T_{\w}(\frac{\D\w}{\w-2})) + b$ and $0 \le b < e(\T_{\w}(\frac{\D\w}{\w-2}))$, and $\CT_{\w}(b)$ is the $\w$-partite colex Tur\'{a}n graph on $b$ edges (see \cite{Frohmader08}).

Second, when $\D$ is not a multiple of $\w-u$ we proved that there exists $c=P_u(H,\D,\w)$ such that $\ex_u(p,H,\set{K_u\vee I_{\D+1},K_{\w+1}})=(1-o_p(1))cp$ and determined the value of $c$ asymptotically, as $\D$ increases without bound. It would be interesting to find some exact values of $c$. 
\begin{problem}\label{problem:2}
    For $H$ a graph containing $u$ dominating vertices and $\D$ not a multiple of $\w-u$, determine the exact value of $c=P_u(H,\D,\w)$ such that $\ex_u(p,H,\set{K_u\vee I_{\D+1},K_{\w+1}})=(1-o_p(1))cp$.
\end{problem}
Even in the special case where $H=K_t$ and $u=1$, \cref{problem:2} is difficult. In contrast to Theorems \ref{thm:zykov} and \ref{thm:CGLS}, where the same graphs maximize the number of $t$-cliques for every $t$, in \cite{KR20} we showed that for some values of $\D$ and $\w$, different values of $t$ yield different extremal graphs for $P_u(K_t,\D,\w)$ and therefore for $\ex(n,K_t,\set{K_{1,\D+1},K_{\w+1}})$. It could be fruitful to consider the case where $u=1$ and $\D = \w$, which was solved for $H=K_t$ in \cite{KR20}, for other graphs $H$.

Third, throughout the paper we assumed that $H$ contains $u$ dominating vertices, but the definition of $\ex_u(p,H,\set{K_u\vee I_{\D+1},K_{\w+1}})$ requires only that $H$ contain $K_u$ as a subgraph, so some graphs $H$ have not been addressed.
\begin{problem}
    Asymptotically determine the constant $c=P_u(H,\D,\w)$ such that $\ex_u(p,H,\set{K_u\vee I_{\D+1},K_{\w+1}})=(1-o_p(1))cp$ when $H$ is a graph that contains $K_u$ as a subgraph but does not contain $u$ dominating vertices. 
\end{problem}

The following generalized Tur\'an problem seems similar to $\ex_u(p,H,\set{K_u\vee I_{\D+1},K_{\w+1}})$ and would be interesting to address.
\begin{problem}\label{problem:star} Asymptotically determine the constant $c=c_u(H,\D,\w)$ such that \[\ex_u(p,H,\set{K_{1,\D+1},K_{\w+1}}) = (1-o_p(1))cp,\] for $u > 1$ and for an arbitrary graph $H$ having at least $u$ dominating vertices.
\end{problem}
For \cref{problem:star} the methods of the present paper can be used to show that \[\rho_u(H,\T_{\w}(\D+\floor[\Big]{\frac{\D}{\w-1}})) \le c \le \frac{\N(H^{\downarrow u}, \T_{\w-u}(\D-u+1))}{\binom{\dom(H)}{u}},\]but these bounds do not appear to coincide.

\section*{Acknowledgments}

The author thanks JD Nir and Jamie Radcliffe for helpful discussions and is partially supported by Simons Foundation Grant MP-TSM-00002688.

\bibliographystyle{plainurl}
\bibliography{references}
\end{document}